\documentclass[11pt,a4paper]{article}

\usepackage{titlesec}
\usepackage{fancyhdr}
\usepackage{a4wide}
\usepackage{graphicx}
\usepackage{float}
\usepackage{amssymb}
\usepackage{amsmath}
\usepackage{amsthm}
\usepackage{color}
\usepackage{mathrsfs}
\usepackage{array}
\usepackage{eucal}
\usepackage{tikz}
\usepackage[T1]{fontenc}
\usepackage{inputenc}
\usepackage[english]{babel}
\usepackage{lmodern}
\usepackage{hyperref}
\usepackage{geometry}
\usepackage{changepage}
\changepage{0pt}{}{}{}{}{0pt}{}{0pt}{10pt}
\usepackage[numbers]{natbib}
\usepackage{hyperref}
\hypersetup{
pdfpagemode=none,
pdftoolbar=true,        
pdfmenubar=true,        
pdffitwindow=false,     
pdfstartview={Fit},    
pdftitle={Exact zero transmission during the Fano resonance phenomenon in non symmetric waveguides},    
pdfauthor={L. Chesnel, S.A. Nazarov},     
pdfsubject={},  
pdfcreator={L. Chesnel, S.A. Nazarov},   
pdfproducer={L. Chesnel, S.A. Nazarov}, 
pdfkeywords={}, 
pdfnewwindow=true,      
colorlinks=true,       
linkcolor=magenta,          
citecolor=red,        
filecolor=cyan,      
urlcolor=blue           
}

\newcommand{\dsp}{\displaystyle}

\newcommand{\eps}{\varepsilon}

\newcommand{\Om}{\Omega}
\newcommand{\mrm}[1]{\mathrm{#1}}

\newcommand{\Cplx}{\mathbb{C}}

\newcommand{\R}{\mathbb{R}}

\newcommand{\mL}{\mrm{L}}
\newcommand{\mH}{\mrm{H}}

\newcommand{\tr}{\mrm{tr}}

\newtheorem{theorem}{Theorem}[section]

\newtheorem{remark}{Remark}[section]

\newtheorem{proposition}{Proposition}[section]

\begin{document}

~\vspace{0.0cm}
\begin{center}
{\sc \bf\LARGE  
Exact zero transmission during the Fano resonance\\[6pt] phenomenon in non symmetric waveguides}
\end{center}

\begin{center}
\textsc{Lucas Chesnel}$^1$, \textsc{Sergei A. Nazarov}$^{2}$\\[16pt]
\begin{minipage}{0.95\textwidth}
{\small
$^1$ INRIA/Centre de mathématiques appliquées, \'Ecole Polytechnique, Université Paris-Saclay, Route de Saclay, 91128 Palaiseau, France;\\
$^2$ St. Petersburg State University, Universitetskaya naberezhnaya, 7-9, 199034, St. Petersburg, Russia;\\[10pt]
E-mails: \texttt{lucas.chesnel@inria.fr}, \texttt{srgnazarov@yahoo.co.uk}, \texttt{s.nazarov@spbu.ru} \\[-14pt]
\begin{center}
(\today)
\end{center}
}
\end{minipage}
\end{center}
\vspace{0.4cm}

\noindent\textbf{Abstract.} 
We investigate a time-harmonic wave problem in a waveguide. We work at low frequency so that only one mode can propagate. It is known that the scattering matrix exhibits a rapid variation for real frequencies in a vicinity of a complex resonance located close to the real axis. This is the so-called Fano resonance phenomenon. And when the geometry presents certain properties of symmetry, there are two different real frequencies such that we have either $R=0$ or $T=0$, where $R$ and $T$ denote the reflection and transmission coefficients. In this work, we prove that without the assumption of symmetry of the geometry, quite surprisingly, there is always one real frequency for which we have $T=0$. In this situation, all the energy sent in the waveguide is backscattered. However in general, we do not have $R=0$ in the process. We provide numerical results to illustrate our theorems.\\

\noindent\textbf{Key words.} Waveguides, Fano resonance, zero transmission, scattering matrix.

\section{Introduction}\label{Introduction}

The Fano resonance is a universal phenomenon in physics which appears in many areas. For a general presentation, we refer the reader to \cite{Fano61} for the seminal paper and to \cite{MiFK10,LZMHN10} for recent reviews. In this work, we consider its expression on a model problem of propagation of time-harmonic waves in a waveguide which is unbounded in one direction. This problem appears naturally for instance in acoustics, in water-waves theory or in electromagnetism. In this context, the Fano resonance mechanism can be described as follows. Assume that the Neumann Laplacian (for the problem we consider below) has a real eigenvalue $\lambda^0$ embedded in the continuous spectrum. In this case, the corresponding eigenfunctions are the so-called trapped modes which are exponentially decaying at infinity. Then perturbing slightly the setting, for example the geometry or the material index, in general this real eigenvalue will turn into a complex resonance \cite{AsPV00,Zwor99,Naza18}. And for real spectral parameters $\lambda$ (proportional to the square of the frequency) varying in a neighbourhood of $\lambda^0$, the scattering matrix will exhibit a rapid variation. This variation is even quicker as the imaginary part of the complex resonance is small. When $\lambda^0$ is between the first and the second thresholds in the continuous spectrum, so that only two conjugated waves can propagate in the waveguide, the symmetric scattering matrix is composed of two reflection coefficients $R$, $\tilde{R}$ and one transmission coefficient $T$ (see the notation in (\ref{defusca})). In this case, under certain properties of symmetry of the configuration, one can show that the scattering coefficients take zero values for some real $\lambda$ around $\lambda^0$. Such particular values for $R$, $\tilde{R}$ are studied in particular in the context of Perfect Transmission Resonances (PTRs), see e.g. \cite{Shao94,PoGP99,LeKi01,Zhuk10,MrMK11}. For the presentation of simple models in optics explaining the Fano resonance phenomenon, we refer the reader to \cite{FaJo02,FaSJ03}. For more mathematical approaches, one can consult \cite{ShVe05,ShTu12,ShWe13,AbSh16,ChNa18}. For computations of complex resonances and numerical investigations of the Fano resonance phenomenon in waveguides, we refer the reader to \cite{DKLM07,CaLo07,EMADAD08,HeKo08,HoNa09,HeKN12}. For results concerning the existence of trapped modes associated with eigenvalues embedded in the continuous spectrum, see e.g. \cite{Urse51,Evan92,EvLV94,DaPa98,LiMc07,Naza10c,Naza13Gaps,Naza17}. Finally, note that another approach to get rigorously a zero transmission coefficient can be found in \cite{ChNPSu,ChPaSu,ChPa19}.\\
\newline
The goal of this note is to show that without assumption of symmetry of the configuration, the transmission coefficient $T$ still takes the zero value throughout the Fano resonance phenomenon. This was intuited in \cite{Lee99}  using a continuation idea from a symmetric setting. In the present work, we prove rigorously the result using a different approach which does not require to start from a symmetric setting. The outline of the article is as follows. First we present the setting in Section \ref{SectionSetting}. Then we perturb the geometry and the frequency of the configuration supporting trapped modes via a small parameter $\eps>0$ and we recall the results of \cite{ChNa18} providing an asymptotic expansion of the scattering matrix with respect to $\eps$ tending to zero. In Section \ref{sectionAsympto}, we show that miraculously (we have no physical explanation for that), the main asymptotic term in the expansion of the transmission coefficient passes through zero for real $\lambda$ around $\lambda^0$. Then in Section \ref{SectionTNull}, working as in \cite{ChPa19}, we demonstrate that the unitary structure of the scattering matrix is enough to deduce that the transmission coefficient itself passes through zero for real $\lambda$ around $\lambda^0$.  We provide some numerical results to illustrate this analysis in Section \ref{SectionNumerics}. Finally, we give short concluding remarks. The main result of this work is Theorem \ref{MainThmPart1}. 

\section{Setting}\label{SectionSetting}

\begin{figure}[!ht]
\centering
\begin{tikzpicture}[scale=2]
\draw[fill=gray!30,draw=none](-1.5,0) rectangle (1.5,1);
\begin{scope}[scale=0.5]
\draw [fill=white,draw=black] plot [smooth cycle, tension=1] coordinates {(-0.6,0.9) (0,0.5) (0.7,1) (0.5,1.5) (-0.2,1.4)};
\end{scope}
\draw (-1.5,1)--(1.5,1);
\draw (-1.5,0)--(1.5,0);
\draw[dashed] (2,1)--(1.5,1);
\draw[dashed] (-2,1)--(-1.5,1);  
\draw[dashed] (2,0)--(1.5,0);
\draw[dashed] (-2,0)--(-1.5,0);  
\node at (0.8,0.15){\small $\Om$};
\begin{scope}[xshift=-6cm,yshift=-1cm]
\draw[->] (3,1.2)--(3.6,1.2);
\draw[->] (3.1,1.1)--(3.1,1.7);
\node at (3.65,1.3){\small $x$};
\node at (3.25,1.6){\small $y$};
\end{scope}
\node[above] at (-1,-0.5){\small $d$};
\draw (-1,-0.1)--(-1,0.1);
\node[above] at (1,-0.5){\small $-d$};
\draw (1,-0.1)--(1,0.1);
\end{tikzpicture}
\caption{Example of geometry $\Om$. \label{DomainOriginal2D}} 
\end{figure}
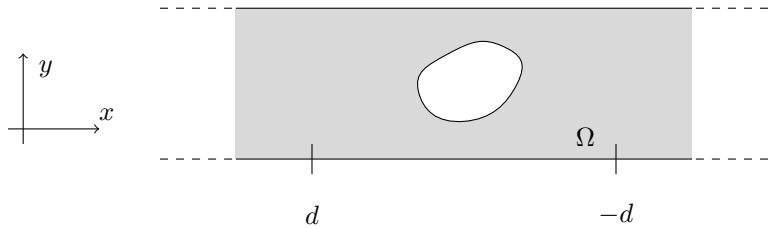

\noindent \noindent Let $\Om\subset\R^2$ be a domain, that is a connected open set, with Lipschitz boundary  $\partial\Om$ which coincides with the reference strip 
\[
\{ (x,y)\in\R\times(0;1)\}
\]
for $|x|\ge d$ where $d>0$ is fixed (see Figure \ref{DomainOriginal2D}). We assume that the propagation of time-harmonic waves in $\Om$ is governed by the Helmholtz equation with Neumann boundary conditions
\begin{equation}\label{PbInitial}
\begin{array}{|rcll}
\Delta u + \lambda u & = & 0 & \mbox{ in }\Om\\[3pt]
 \partial_{\nu}u  & = & 0  & \mbox{ on }\partial\Om. 
\end{array}
\end{equation}
In this problem, $u$ is the quantity of interest (acoustic pressure, velocity potential, component of the electromagnetic field,...), $\Delta$ denotes the 2D Laplace operator, $\lambda$ is a parameter which is proportional to the square of the frequency and $\nu$ stands for the normal unit vector to $\partial\Om$ directed to the exterior of $\Om$. Note that from time to time, abusively we will call $\lambda$ the frequency. We emphasize that we consider an academic $\mrm{2D}$ problem only to simplify the presentation. Other configurations can be dealt with in a completely similar way. In particular, the analysis is the same in higher dimension and in waveguides for which the two unbounded branches are not aligned. Moreover, we can also impose Dirichlet or periodic boundary conditions in (\ref{PbInitial}) to study quantum waveguides or gratings. For $\lambda\in(0;\pi^2)$, only the plane waves $w_{\pm}$ defined by
\begin{equation}\label{DefPlaneWaves}
w_{\pm}(x,y)=e^{\pm ix\sqrt{\lambda} }
\end{equation}
can propagate in $\Om$. For $\lambda\in(0;\pi^2)$, the problem (\ref{PbInitial}) has solutions $u_{\pm}$ admitting the decompositions 
\begin{equation}\label{defusca}
u_+=\begin{array}{|ll}
w_++R_+\,w_-+\dots,& \mbox{ for }x<-d\\
\hspace{1.1cm}T\ w_++\dots,& \mbox{ for }x>d,
\end{array}\qquad\qquad u_-=\begin{array}{|ll}
\hspace{1cm}T\ w_-+\dots,& \mbox{ for }x<-d\\
w_-+R_-\,w_++\dots,& \mbox{ for }x>d.
\end{array}
\end{equation}
Here $R_{\pm}\in\Cplx$ are reflection coefficients and $T\in\Cplx$, which is the same both for $u_+$ and $u_-$ due to the reciprocity relation, is the transmission coefficient. Moreover, the dots stand for remainders in $\mH^1(\Om)$ which decay as $O(e^{-|x|\sqrt{\pi^2-\lambda}})$ when $|x|\to+\infty$. Physically, $u_+$ (resp. $u_-$) models the scattering of the incident rightgoing wave $w_+$ (resp. leftgoing wave $w_-$) by the perturbation in the geometry with respect to the reference strip $\R\times(0;1)$. We define the scattering matrix
\[
\mathfrak{s}:=\left( \begin{array}{cc}
R_+ & T \\
T & R_- \\
\end{array}\right)\in\Cplx^{2\times2}.
\]
It is a classical exercise to show that $\mathfrak{s}$ is unitary ($\mathfrak{s}\overline{\mathfrak{s}}^{\top}=\mrm{Id}$) and symmetric ($\mathfrak{s}=\mathfrak{s}^{\top}$). The functions $u_{\pm}$ are uniquely defined if and only if trapped modes (non-zero solutions of (\ref{PbInitial}) which are in $\mL^2(\Om)$) do not exist at the chosen $\lambda$. If trapped modes exist, we define uniquely $u_{\pm}$ as the functions admitting the expansions (\ref{defusca}) and which are orthogonal to the linear space of trapped modes (which is of finite dimension) in $\mL^2(\Om)$.\\
\newline
We assume that the geometry $\Om$ is such that $\lambda=\lambda^0\in(0;\pi^2)$ is a simple eigenvalue of the Neumann Laplacian. In other words, we assume there is a non zero $u_{\tr}\in\mL^2(\Om)$ satisfying $\Delta u_{\tr} + \lambda^0 u_{\tr}  = 0$ in $\Om$, $ \partial_{\nu}u_{\tr}  =  0$ on $\partial\Om$ and that any $\mL^2$ solution of (\ref{PbInitial}) is proportional to $u_{\tr}$. Note that since the continuous spectrum of the Neumann Laplacian in $\Om$ is $\sigma_c=[0;+\infty)$, the eigenvalue is embedded in $\sigma_c$. To set ideas, we impose that $\|u_{\tr}\|_{\mL^2(\Om)}=1$. Using decomposition in Fourier series, we obtain the expansion 
\begin{equation}\label{DefTrappedMode}
u_{\tr}= K\,e^{-x\sqrt{\pi^2-\lambda^0} }\cos(\pi y)+\tilde{u}_{\tr}\quad\mbox{ for } x\ge d,
\end{equation}
where $K$ is a constant and $\tilde{u}_{\tr}$ is a remainder which decays as $O(e^{-x\sqrt{4\pi^2-\lambda^0} })$ when $x\to+\infty$. We assume that $u_{\tr}$ has a slow decay as $x\to+\infty$, i.e. $K\ne0$. In case $K=0$, the analysis below must be adapted but can be done. Without lost of generality, we can impose that $K>0$. Note that the choice of making an assumption on the decay of $u_{\tr}$ as $x\to+\infty$ is arbitrary. Considering the change $x\mapsto-x$, the analysis below can be developed completely similarly imposing the behaviour as $x\to -\infty$. 

\section{Perturbation of the frequency and of the geometry}\label{SectionPerturbed}

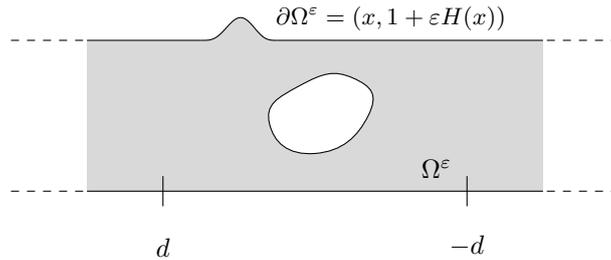
\begin{figure}[!ht]
\centering
\begin{tikzpicture}[scale=2]
\draw[fill=gray!30,draw=none](-1.5,0) rectangle (1.5,1);
\begin{scope}[scale=0.5]
\draw [fill=white,draw=black] plot [smooth cycle, tension=1] coordinates {(-0.6,0.9) (0,0.5) (0.7,1) (0.5,1.5) (-0.2,1.4)};
\end{scope}
\draw (-1.5,1)--(-0.75,1);
\draw (-0.25,1)--(1.5,1);
\draw (-1.5,0)--(1.5,0);
\draw[dashed] (2,1)--(1.5,1);
\draw[dashed] (-2,1)--(-1.5,1);  
\draw[dashed] (2,0)--(1.5,0);
\draw[dashed] (-2,0)--(-1.5,0);  
\node at (0.8,0.15){\small $\Om^{\eps}$};
\node[above] at (-1,-0.5){\small $d$};
\draw (-1,-0.1)--(-1,0.1);
\node[above] at (1,-0.5){\small $-d$};
\draw (1,-0.1)--(1,0.1);
\draw[samples=30,domain=-0.25:0.25,draw=black,fill=gray!30] plot(\x-0.5,{1+0.05*(4*\x+1)^4*(4*\x-1)^4*(4*\x+3)});
\node[above] at (0.5,1){\footnotesize  $\partial \Om^{\eps}=(x,1+\eps H(x))$};
\end{tikzpicture}
\caption{Example of perturbed waveguide $\Om^{\eps}$. \label{PerturbedHalfWaveguide}} 
\end{figure}

\noindent Now, we perturb slightly the original setting supporting trapped modes. First, the spectral parameter $\lambda^0$ is changed for
\begin{equation}\label{defPerturbFreq}
\lambda^{\eps}=\lambda^0+\eps\lambda'
\end{equation}
where $\lambda'\in\R$ is given and $\eps>0$ is small. Second, we make a perturbation of amplitude $\eps$ of the geometry to change $\Om$ into some new waveguide $\Om^{\eps}$. More precisely, consider $\gamma\subset\partial\Om$ a smooth arc. In a neighbourhood $\mathscr{V}$ of $\gamma$, we introduce natural curvilinear coordinates $(n,s)$ where $n$ is the oriented distance to $\gamma$ such that $n>0$ outside $\Om$ and $s$ is the arc length on $\gamma$. Additionally, let $H\in\mathscr{C}^{\infty}_0(\gamma)$ be a smooth profile function which vanishes in a neighbourhood of the two endpoints of $\gamma$. Outside $\mathscr{V}$, we assume that $\partial\Om^{\eps}$ coincides with $\partial\Om$ and inside $\mathscr{V}$,  $\partial\Om^{\eps}$ is defined by the equation
\begin{equation}\label{DefGeom}
n(s)=\eps H(s).
\end{equation}
In other words if $\gamma$ is parametrized as $\gamma=\{P(s)\in\R^2\,|\,s\in I\}$ where $I$ is a given interval of $\R$, then $\gamma^{\eps}:=\{P(s)+\eps H(s)\nu(s)\,|\,s\in I\}$. Here $\nu(s)$ is the unit vector normal to $\gamma$ at point $P(s)$ directed to the exterior of $\Om$.  Finally we consider the perturbed problem 
\begin{equation}\label{PbInitialHalfPerturbed}
\begin{array}{|rcll}
\Delta u^{\eps} + \lambda^{\eps} u^{\eps} & = & 0 & \mbox{ in }\Om^{\eps}\\[3pt]
 \partial_{\nu^{\eps}}u^{\eps}  & = & 0  & \mbox{ on }\partial\Om^{\eps},
\end{array}
\end{equation}
where $\nu^{\eps}$ stands for the normal unit vector to $\partial\Om^{\eps}$ directed to the exterior of $\Om^{\eps}$. We denote by 
\[
\mathfrak{s}(\eps,\lambda),\qquad T(\eps,\lambda),\qquad R_+(\eps,\lambda),\qquad R_-(\eps,\lambda)
\]
the scattering parameters introduced in the previous section in the geometry $\Om^{\eps}$ at frequency $\lambda$. And for short, we set 
\[
\mathfrak{s}^0:=\mathfrak{s}(0,\lambda^0),\qquad T^0:=T(0,\lambda^0),\qquad R_+^0:=R_+(0,\lambda^0),\qquad R_-^0:=R_-(0,\lambda^0).
\]
To recall the Theorem 5.1 of \cite{ChNa18} describing the behaviour of the scattering matrix $\mathfrak{s}(\eps,\lambda^0+\eps\lambda')$ as $\eps$ goes to zero, and which will be the basis of our analysis below, we need to introduce a few quantities. Set $U:=(u_+,u_-)$ where $u_{\pm}$ are the functions introduced in (\ref{defusca}) for $\lambda=\lambda^0$. Set also 
\begin{eqnarray}
\label{defKappa}\kappa(H)&\hspace{-0.2cm}:=&\hspace{-0.2cm}\dsp\int_{I}H(s)(|\partial_su_{\tr}(0,s)|^2-\lambda^0\,|u_{\tr}(0,s)|^2)\,ds\in\R,\\[4pt]
\label{defAlpha}\alpha&\hspace{-0.2cm}:=&\hspace{-0.2cm}\dsp\int_{\Om}u_{\tr}(x,y)\overline{U(x,y)}\,dxdy\in\Cplx\times\Cplx,\\[4pt]
\label{defBeta}\beta(H)&\hspace{-0.2cm}:=&\hspace{-0.2cm}\dsp\int_{I}H(s)(\partial_su_{\tr}(0,s)\overline{\partial_sU(0,s)}-\lambda^0\,u_{\tr}(0,s)\overline{U(0,s)})\,ds\in\Cplx\times\Cplx.
\end{eqnarray}
\begin{theorem}\label{MainThmAsympto}
$\star$ Assume that $\lambda'\ne\kappa(H)$. Then we have
\[
\lim_{\eps\to0}\ \mathfrak{s}(\eps,\lambda^0+\eps\lambda')=\mathfrak{s}^0.
\]
$\star$ Assume that $H$ is such that $\kappa(H)\alpha\ne\beta(H)\in\Cplx\times\Cplx$. Then we have
\[
\lim_{\eps\to0}\ \mathfrak{s}(\eps,\lambda^0+\eps\kappa(H)+\eps^2\mu)=\mathfrak{s}^0+\cfrac{\tau^{\top}\tau}{i\tilde{\mu}-|\tau|^2/2},
\]
with $\tau:=(\kappa(H)\alpha-\beta(H))\,\mathfrak{s}$ and $\tilde{\mu}:=A\mu+B$ for some unimportant real constants $A$, $B$ with $A\ne0$. We emphasize that $A$, $B$ are independent of $\eps$, $\mu$.  
\end{theorem}
\noindent Let us comment this result. The methodology to prove it is the following. First, we compute an asymptotic expansion of an auxiliary object called the augmented scattering matrix, which has been introduced in \cite{NaPl94bis,KaNa02} and \cite{Naza11,Naza12} as $\eps\to0$. The essential property is that this augmented scattering matrix considered as a function of $(\eps,\lambda)$ is smooth at $(0,\lambda^0)$. The procedure and the proof of error estimates are detailed in \cite{Naza11c,Naza12,Naza13}. Then using the relation existing between the usual scattering matrix and the augmented scattering matrix, we can get the statement of the theorem.\\
\newline
\begin{figure}[!ht]
\centering
\begin{tikzpicture}[scale=0.95]
\draw[->] (-2.9,0) -- (3.1,0) node[right] {$\eps$};
\draw[->] (0,-0.2) -- (0,3) node[right] {$\lambda$};
\node at (0,3.14/2-0.1) [left] {$\lambda^0$};
\begin{scope}
\clip(-2.5,-0.5) rectangle (2.5,3);
\draw[domain=-2.7:2.7,smooth,variable=\x,blue] plot ({\x},{3.14/2-\x*3.14/4+0.2*\x*\x});
\draw[domain=-2.7:2.7,smooth,variable=\x,blue] plot ({\x},{3.14/2-\x*3.14/4+0.3*\x*\x});
\draw[domain=-2.7:2.7,smooth,variable=\x,blue] plot ({\x},{3.14/2-\x*3.14/4+0.1*\x*\x});
\draw[domain=-2.7:2.7,smooth,variable=\x,blue] plot ({\x},{3.14/2-\x*3.14/4+0.05*\x*\x});
\draw[domain=-2.7:2.7,smooth,variable=\x,blue] plot ({\x},{3.14/2-\x*3.14/4-0.05*\x*\x});
\draw[domain=-2.7:2.7,smooth,variable=\x,blue] plot ({\x},{3.14/2-\x*3.14/4-0.2*\x*\x});
\draw[domain=-2.7:2.7,smooth,variable=\x,blue] plot ({\x},{3.14/2-\x*3.14/4-0.3*\x*\x});
\draw[domain=-2.7:2.7,smooth,variable=\x,blue] plot ({\x},{3.14/2-\x*3.14/4-0.1*\x*\x});
\end{scope}
\draw[red,dashed,very thick] (0.6,-0.2) -- (0.6,2.8);
\node at (0.6,-0.2) [below] {\textcolor{red}{$\eps_0$}};
\end{tikzpicture}\caption{The limit of $\mathfrak{s}(\eps,\lambda^{0}+\eps\kappa(H)+\eps^2\mu)$ as $\eps\to0$ depends on the parabolic curve chosen for the frequency. \label{FigFano1DParabola}}
\end{figure}
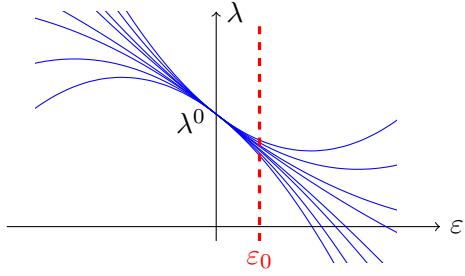

\noindent As explained in \cite{ChNa18}, Theorem \ref{MainThmAsympto} shows that the scattering matrix $\mathfrak{s}(\cdot,\cdot)$ is not continuous at the point $(0,\lambda^0)$ (setting where trapped modes exist). Indeed, the function $\mathfrak{s}(\cdot,\cdot)$ valued on different parabolic paths $\{(\eps,\lambda^{0}+\eps\kappa(H)+\eps^2\mu),\,\eps\in(0;\eps_0)\}$ (see Figure \ref{FigFano1DParabola}) have different limits when $\eps$ tends to zero. And for $\eps_0\ne0$ small fixed, the usual scattering matrix $\lambda\mapsto\mathfrak{s}(\eps_0,\lambda)$ exhibits a quick change in a neighbourhood of $\lambda^0+\eps_0\kappa(H)$. Indeed, the map $\mu\mapsto\mathfrak{s}(\eps_0,\lambda^0+\eps_0\kappa(H)+\eps_0^2\mu)$ has a large variation for $\mu\in[-C\eps_0^{-1};C\eps_0^{-1}]$ for some arbitrary $C>0$ (which is only a small change for $\lambda^{\eps_0}$). Said differently, a change of order $\eps$ of the frequency leads to a change of order one of the scattering matrix. This is nothing but the the Fano resonance phenomenon. For a given $C>0$, outside an interval of length $C\eps_0$ centred at $\lambda^0+\eps_0\kappa(H)$, $\mathfrak{s}(\eps_0,\cdot)$ is approximately equal to $\mathfrak{s}^0$.

\begin{remark}
When $H$ is such that $\kappa(H)\alpha=\beta(H)\in\Cplx\times\Cplx$, in general a fast Fano resonance phenomenon appears. More precisely, for a given $\eps_0\ne0$ small, the variation of $\mathfrak{s}(\eps_0,\cdot)$ of order one occurs on a range of frequencies of length $O(\eps_0^2)$ (instead of $O(\eps_0)$ when $\beta(H)\ne\kappa(H)t$). We write ``in general'' because we can also show that for well-chosen geometric perturbations, obtained solving a fixed-point problem, no Fano resonance phenomenon happens and the real eigenvalue embedded in the continuous spectrum keeps this property instead of becoming a complex resonance. In particular, this latter result allows one to construct non symmetric waveguides with eigenvalues embedded in the continuous spectrum (see \cite{Naza12,Naza13}).
\end{remark}
 
\noindent From now, we denote by $\tau_1$, $\tau_2\in\Cplx$ the two components of $\tau$, so that $\tau=(\tau_1,\tau_2)$, and we set
\[
\begin{array}{lcl}
\mathfrak{s}^{\eps}(\mu)&\hspace{-0.2cm}:=&\hspace{-0.2cm}\mathfrak{s}(\eps,\lambda^0+\eps\kappa(H)+\eps^2\mu)\\[2pt]
T^{\eps}(\mu)&\hspace{-0.2cm}:=&\hspace{-0.2cm}T(\eps,\lambda^0+\eps\kappa(H)+\eps^2\mu)\\[2pt]
R_+^{\eps}(\mu)&\hspace{-0.2cm}:=&\hspace{-0.2cm}R_+(\eps,\lambda^0+\eps\kappa(H)+\eps^2\mu)\\[2pt]
R_-^{\eps}(\mu)&\hspace{-0.2cm}:=&\hspace{-0.2cm}R_-(\eps,\lambda^0+\eps\kappa(H)+\eps^2\mu).
\end{array}
\] 
With this notation, the analysis developed in \cite{ChNa18} provides the estimate
\begin{equation}\label{MainAsymptos}
|\mathfrak{s}^{\eps}(\mu)-s^{\mrm{asy}}(\mu)|\le C \eps\qquad\mbox{with }\qquad s^{\mrm{asy}}(\mu)=\mathfrak{s}^0+\cfrac{\tau^{\top}\tau}{i\tilde{\mu}-|\tau|^2/2},
\end{equation}
where in (\ref{MainAsymptoT}), for any compact set $I\subset\R$, the constant $C>0$ can be chosen independent of $\mu\in I$. In particular, we have
\begin{equation}\label{MainAsymptoT}
|T^{\eps}(\mu)-T^{\mrm{asy}}(\mu)|\le C \eps\qquad\mbox{with }\qquad T^{\mrm{asy}}(\mu)=T^0+\cfrac{\tau_1\tau_2}{i\tilde{\mu}-(|\tau_1|^2+|\tau_2|^2)/2}.
\end{equation}
In order to prove that we have $T^{\eps}(\mu)=0$ for some $\mu\in\R$ for $\eps$ small enough, we first show that the map $\mu\mapsto T^{\mrm{asy}}(\mu)$ vanishes in $\R$. This is the object of the next section. 

\section{Asymptotic behaviour of the transmission coefficient}\label{sectionAsympto}

\begin{proposition}\label{PropositionAsympto}
Assume that $T^0\ne0$. Then we have 
\[
\{T^{\mrm{asy}}(\mu),\,\mu\in\R\}=\mathscr{C}^{\mrm{asy}}\setminus\{T^0\}
\]
where $\mathscr{C}^{\mrm{asy}}$ is a circle passing through $T^0$ and zero. 
\end{proposition}
\begin{proof}
Classical results concerning the M\"{o}bius transform guarantee that $\{T^{\mrm{asy}}(\mu),\,\mu\in\R\}$ coincides with $\mathscr{C}^{\mrm{asy}}\setminus\{T^0\}$ where $\mathscr{C}^{\mrm{asy}}$ is a circle passing through $T^0$. Let us show that $\mathscr{C}^{\mrm{asy}}$ also passes through zero. From (\ref{MainAsymptoT}), one finds that $T^{\mrm{asy}}(\mu)=0$ for some $\mu\in\R$ if and only if there holds
\begin{equation}\label{MainIdentity}
\cfrac{|\tau_1|^2+|\tau_2|^2}{2}=\Re e\left(\cfrac{\tau_1\tau_2}{T^0}\right).
\end{equation}
In order to establish (\ref{MainIdentity}), we need to derive some relations between $T^0$ and $\tau=(\tau_1,\tau_2)$. To proceed, first we notice that $U=(u_+,u_-)$ satisfies
\begin{equation}\label{jolieRelation}
\overline{U}\mathfrak{s}^0=U.
\end{equation}
Indeed, the first component of $\overline{U}\mathfrak{s}^0$ is equal to $R_+^0\,\overline{u_+}+T^0\,\overline{u_-}$ and using (\ref{defusca}), one finds that this function admits the expansion
\[
R_+^0\,\overline{u_+}+T^0\,\overline{u_-}=\begin{array}{|ll}
R_+^0\,(w_-+\overline{R_+^0}\,w_+)+T^0\,(\overline{T^0}\,w_+)+\dots,& \mbox{ for }x<-d\\
R_+^0\,(\overline{T^0}\,w_-)+T^0\,(w_++\overline{R_-^0}\,w_-)+\dots,& \mbox{ for }x>d.
\end{array}
\]
From the unitarity of $\mathfrak{s}^0$, we infer that $R_+^0\,\overline{u_+}+T^0\,\overline{u_-}$ has the same expansion as $u_{+}$ at infinity. Using that $u_{\pm}$ are orthogonal to $u_{\tr}$ in $\mL^2(\Om)$, we deduce that $R_+^0\,\overline{u_+}+T^0\,\overline{u_-}=u_+$. Similarly, we show that $T^0\,\overline{u_+}+R_-^0\,\overline{u_-}=u_-$, which allows us to conclude to (\ref{jolieRelation}). Now we exploit (\ref{jolieRelation}) to establish the identity
\begin{equation}\label{identityUtile}
\overline{\tau}\mathfrak{s}^0=\tau.
\end{equation}
From the expressions (\ref{defKappa})-(\ref{defBeta}) of $\kappa(H)$, $\alpha$, $\beta(H)$ and the properties of $\mathfrak{s}^0$, we obtain
\[
\overline{\tau}\mathfrak{s}^0=(\kappa(H)\overline{\alpha}-\overline{\beta(H)})\,\overline{\mathfrak{s}^0}\mathfrak{s}^0=\kappa(H)\overline{\alpha}-\overline{\beta(H)}.
\]
Then replacing $U$ by $\overline{U}\mathfrak{s}^0$ (identity (\ref{jolieRelation})) in $\kappa(H)\overline{\alpha}-\overline{\beta(H)}$, we get $\overline{\tau}\mathfrak{s}^0=(\kappa(H)\alpha-\beta(H))\mathfrak{s}^0=\tau$. This proves (\ref{identityUtile}) or equivalently
\begin{equation}\label{FormulaAsympto}
\begin{array}{|ll}
R_+^0\,\overline{\tau_1}+T^0\,\overline{\tau_2}=a\\[4pt]
T^0\,\overline{\tau_1}+R_-^0\,\overline{\tau_2}=b.
\end{array}
\end{equation}
Finally, we use (\ref{FormulaAsympto}) to establish (\ref{MainIdentity}). The unitarity of $\mathfrak{s}^0$ imposes $R_-^0=-\overline{R_+^0} T^0/\overline{T^0}$. Inserting this relation in the second line of  (\ref{FormulaAsympto}) gives 
\begin{equation}\label{Eqn1}
T^0\,\overline{\tau_1}-\cfrac{\overline{R_+^0} T^0}{\overline{T^0}}\,\overline{\tau_2}=\tau_2.
\end{equation}
The first line of (\ref{FormulaAsympto}) implies 
\begin{equation}\label{Eqn2}
R_+^0=\frac{\tau_1-T^0\overline{\tau_2}}{\overline{\tau_1}}.
\end{equation}
Inserting (\ref{Eqn2}) in (\ref{Eqn1}) and multiplying by $\tau_1$ leads to 
\[
T^0\,(|\tau_1|^2+|\tau_2|^2)-\cfrac{ T^0}{\overline{T^0}}\,\overline{\tau_1\tau_2}=\tau_1\tau_2\qquad\Leftrightarrow \qquad |\tau_1|^2+|\tau_2|^2 = 2 \,\Re e\left(\cfrac{\tau_1\tau_2}{T^0}\right).
\]
This is identity (\ref{MainIdentity}).
\end{proof}
\begin{remark}
The reason why $\mathscr{C}^{\mrm{asy}}$ passes through zero is quite mysterious. When $\Om$, $\Om^{\eps}$ are symmetric with respect to the $(Oy)$ axis, this can be shown quite simply working with half-waveguides problems (see e.g. \cite{ChNa18}). But without assumption of symmetry, we cannot provide a physical interpretation of this fact.
\end{remark}
\noindent Denote $\mu_{\star}$ the value of $\mu$ such that $T^{\mrm{asy}}(\mu_{\star})=0$ and for $\eps>0$, define the interval $I^{\eps}=(\mu_{\star}-\sqrt{\eps};\mu_{\star}+\sqrt{\eps})$.
From (\ref{MainAsymptoT}), for $\eps>0$ small, we know that the curve
\[
C^{\eps}=\{T^{\eps}(\mu),\,\mu\in I^{\eps}\}
\]
passes close to zero. It remains to show that $C^{\eps}$ passes exactly through zero for $\eps$ small enough.

\section{Exact zero transmission}\label{SectionTNull}

Now, we state and prove the main result of the article. Its proof relies on Proposition \ref{PropositionAsympto} and an argument presented in \cite{ChPa19} (see also \cite{Lee99}).

\begin{theorem}\label{MainThmPart1}
Assume that $T^0\ne0$. Then there is $\eps_0>0$ such that for all $\eps\in(0;\eps_0]$, there exists $\mu\in\R$ such that $T^\eps(\mu)=0$.
\end{theorem}
\begin{proof}
Let us first give the general idea of the proof. Assume by contradiction that for all $\eps>0$, $\mu\mapsto T^{\eps}(\mu)$ does not pass through zero in $I^\eps$. Since $\mathfrak{s}^{\eps}(\mu)$ is unitary, there holds $R_+^{\eps}(\mu)\,\overline{T^{\eps}(\mu)}+T^{\eps}(\mu)\,\overline{R_-^{\eps}(\mu)}=0$ and so
\begin{equation}\label{Unitarity}
-R_+^{\eps}(\mu)/\overline{R_-^{\eps}(\mu)}=T^{\eps}(\mu)/\overline{T^{\eps}(\mu)}\qquad\forall\mu\in I^{\eps}.
\end{equation}
But if $\mu\mapsto T^{\eps}(\mu)$ does not pass through zero on $I^{\eps}$, using Proposition \ref{PropositionAsympto} one can verify that the point $T^{\eps}(\mu)/\overline{T^{\eps}(\mu)}=e^{2i\mrm{arg}(T^{\eps}(\mu))}$ must run rapidly on the unit circle for $\mu\in I^\eps$ as $\eps\to0$. On the other hand, $R_+^{\eps}(\mu)/\overline{R_-^{\eps}(\mu)}$ tends to a constant in $I^\eps$ as $\eps\to0$. This way we obtain a contradiction. We emphasize that the unitary structure of $\mathfrak{s}^{\eps}(\mu)$ is the key ingredient of this step of the proof. Now we make this discussion more rigorous.\\

\begin{figure}[!ht]
\centering
\begin{tikzpicture}[scale=3]
\draw[draw=none,fill=red!30] (0,0)--(8.13:1)--(8.13:1)arc(8.13:98.13:1)--cycle;
\draw[draw=none,fill=red!30] (0,0)--(188.13:1)--(188.13:1)arc(188.13:278.13:1)--cycle;
\draw[blue,line width=0.3mm] (4/5,-3/5) circle (5/5);
\draw[green!60!black,line width=0.3mm] (4/5,-3/5+0.1) ellipse (1.1 and 0.9);
\draw[->] (-1.2,0) -- (1.2,0);
\draw[->] (0,-1.2) -- (0,1.2);
\draw[red,line width=0.3mm] (0,0)--(8.13:1);
\draw[red,line width=0.3mm] (0,0)--(98.13:1);
\draw[red,line width=0.3mm] (0,0)--(188.13:1);
\draw[red,line width=0.3mm] (0,0)--(278.13:1);
\draw[domain=-0.8:0.8,smooth,variable=\x] plot ({\x},{4*\x/3.});
\node[blue] at (1.4,-1.55){$\mathscr{C}^{\mrm{asy}}$};
\node[magenta] at (-0.26,0.17){$T^{\eps}(a_{\eps})$};
\node[magenta] at (-0.3,-0.15){$T^{\eps}(b_{\eps})$};
\node[green!60!black] at (2.2,-0.2){$\mu\mapsto T^{\eps}(\mu)$};
\node at (-0,0.8){$Q_{1}$};
\node at (0,-0.8){$Q_{2}$};
\fill[magenta] (-0.015,0.11) circle (0.02);
\fill[magenta] (-0.13,-0.02) circle (0.02);
\node at (1.35,1.1){$\{\rho\,e^{i\eta}\in\Cplx,\ \rho\in\R\}$};
\end{tikzpicture}
\caption{Notation used in the proof of Theorem \ref{MainThmPart1}.\label{PictureProof}}
\end{figure}
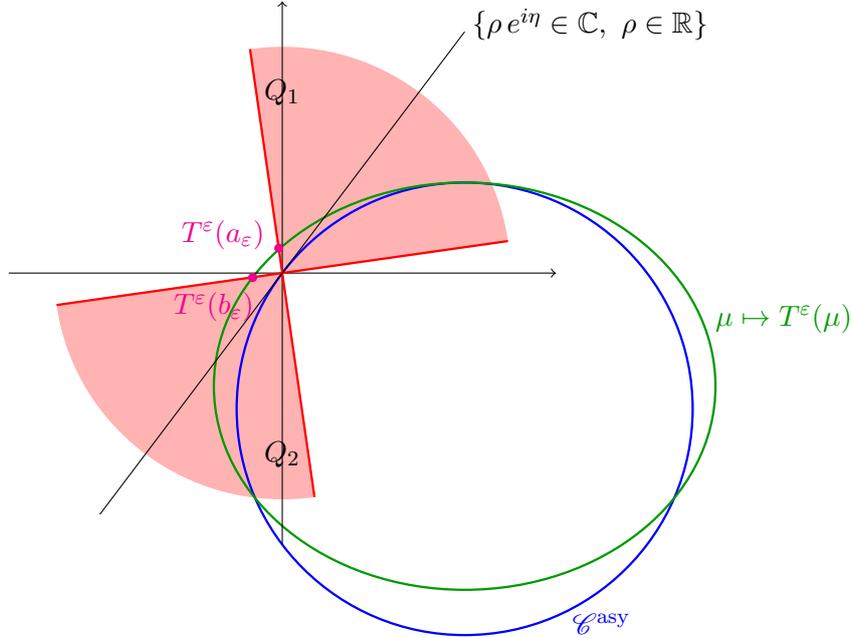

\noindent Since the circle $\mathscr{C}^{\mrm{asy}}$ passes through zero, there is $\eta\in(-\pi/2;\pi/2]$ such that $\mathscr{C}^{\mrm{asy}}$ is tangent to the line $\{\rho\,e^{i\eta}\in\Cplx,\ \rho\in\R\}$. Define the quadrants 
\[
\begin{array}{ll}
Q_1:=\{\rho\,e^{i\theta}\in\Cplx\,|\,\rho>0,\ \eta-\pi/4<\theta<\eta+\pi/4\}\\[2pt]
Q_2:=\{\rho\,e^{i\theta}\in\Cplx\,|\,\rho<0,\ \eta-\pi/4<\theta<\eta+\pi/4\},
\end{array}
\]
see Figure \ref{PictureProof}. The graph of the map $\mu\mapsto T^{\mrm{asy}}(\mu)$ crosses both quadrants $Q_1$ and $Q_2$ in $I_{\eps}$. On the other hand, we have $|T^{\eps}(\mu)-T^\mrm{asy}(\mu)|\le C\eps$ where $C>0$ is independent of $\mu\in I_{\eps}$ for all $\eps\in(0;\eps_0]$. As a consequence, there is $\eps_0$ such that for all $\eps\in(0;\eps_0]$, the graph of the map $\mu\mapsto T^{\eps}(\mu)$ intersects both $Q_1$ and $Q_2$ on $I_{\eps}$.\\
\newline
If  $\mu\mapsto T^{\eps}(\mu)$ does not vanish in $I^\eps$, since $\mu\mapsto T^{\eps}(\mu)$ is continuous, we deduce that for all $\eps\in(0;\eps_0]$, there are $a_{\eps},\,b_{\eps}\in I^\eps$ such that $T^{\eps}(a_{\eps})=t_{\eps}\,e^{i(\eta-\pi/4)}$ and $T^{\eps}(b_{\eps})=\tilde{t}_{\eps}\,e^{i(\eta+\pi/4)}$, with $t_{\eps}$, $\tilde{t}_{\eps}\in\R\setminus\{0\}$. Taking successively $\mu=a_{\eps}$, $\mu=b_{\eps}$ in the relation preceding (\ref{Unitarity}), we obtain 
\begin{equation}\label{RelationContradicton}
R_+^{\eps}(a_{\eps})=-ie^{2i\eta}\overline{R_-^{\eps}(a_{\eps})}\qquad\mbox{ and }\qquad R_+^{\eps}(b_{\eps})=ie^{2i\eta}\overline{R_-^{\eps}(b_{\eps})}. 
\end{equation}
Introduce the functions $R_{\pm}^{\mrm{asy}}$ such that
\[
R_+^{\mrm{asy}}(\mu)=R_+^0+\cfrac{a^2}{i\tilde{\mu}-(|a|^2+|b|^2)/2}\qquad\mbox{ and }\qquad R_-^{\mrm{asy}}(\mu)=R_-^0+\cfrac{b^2}{i\tilde{\mu}-(|a|^2+|b|^2)/2}.
\]
From (\ref{MainAsymptos}), we know that there is $\eps_0>0$ such that, for all $\eps\in(0;\eps_0]$, we have 
\[
R_+^{\eps}(a_{\eps}),\,R_+^{\eps}(b_{\eps})\in B(R_+^{\mrm{asy}}(\mu_{\star}),\eps^{1/4})\qquad\mbox{ and }\qquad R_-^{\eps}(a_{\eps}),\,R_-^{\eps}(b_{\eps})\in B(R_-^{\mrm{asy}}(\mu_{\star}),\eps^{1/4}),
\]
where for $z_0\in\Cplx$, $B(z_0,r)$ denotes the open disk of $\Cplx$ of radius $r>0$ centred at $z_0$. From (\ref{RelationContradicton}), we deduce that we must have both
\[
B(R_+^{\mrm{asy}}(\mu_{\star}),\eps^{1/4})\cap B(ie^{2i\eta}\overline{R_-^{\mrm{asy}}(\mu_{\star})},\eps^{1/4})\neq\emptyset\ \mbox{ and }\  B(R_+^{\mrm{asy}}(\mu_{\star}),\eps^{1/4})\cap B(-ie^{2i\eta}\overline{R_-^{\mrm{asy}}(\mu_{\star})},\eps^{1/4})\neq\emptyset.
\]
This is impossible for $\eps$ small enough because $|R_-^{\mrm{asy}}(\mu_{\star})|=1$ (remember that $T^{\mrm{asy}}(\mu_{\star})=0$). Thus, we deduce that for all $\eps\in(0;\eps_0]$, $\mu\mapsto T^{\eps}(\mu)$ cancels in $I^\eps$. 
\end{proof}

\noindent Concerning the zeros of $\mu\mapsto R_+^{\eps}(\mu)$, we can make the following comments. When $\eps$ tends to zero, from (\ref{MainAsymptos}), we know that the curve $\{R_+^{\eps}(\mu),\,\mu\in\R\}$ gets closer and closer to $\{R_+^{\mrm{asy}}(\mu),\,\mu\in\R\}$. The set $\{R_+^{\mrm{asy}}(\mu),\,\mu\in\R\}$ is a circle. It passes through zero if and only if we have
\begin{equation}\label{MainIdentity}
\cfrac{|a|^2+|b|^2}{2}=\Re e\left(\cfrac{a^2}{R_+^0}\right).
\end{equation}
Dividing the first line of (\ref{FormulaAsympto}) by $R_+^0$ and computing the square of the modulus, we obtain the identity
\[
|a|^2\left(1+\cfrac{1}{|R_+^0|^2}\right)-2\,\Re e\,\left(\cfrac{a^2}{R_+^0}\right)=\cfrac{|T^0|^2}{|R_+^0|^2}\,|b|^2.
\]
Using the above equality, we obtain that (\ref{MainIdentity}) is satisfied if and only if there holds 
\begin{equation}\label{Criterion}
|a|=|b|.
\end{equation}
As a consequence, if $|a|\ne|b|$, for $\eps$ small enough, $\mu\mapsto R_+^{\eps}(\mu)$ does not pass through zero. Using the definition of $\tau$ in Theorem \ref{MainThmAsympto}, we observe that we have $|a|=|b|$ if $\Om$ and $H$ are symmetric with respect to the $(Oy)$ axis. However, surely it is not necessary to consider symmetric geometries to have (\ref{Criterion}). But we emphasize that if (\ref{Criterion}) holds in a non symmetric setting, then we cannot work as in the proof of Theorem \ref{MainThmPart1} to get exactly $R_+^{\eps}(\mu)=0$ for some $\mu\in\R$. Everything lies in the fact that the identity (\ref{Unitarity}) cannot be exploited similarly for the reflection and the transmission coefficients. Therefore, a priori nothing guarantees that exact zero reflection occurs during the Fano resonance phenomenon in a non symmetric waveguide, even when (\ref{Criterion}) is satisfied.

\section{Numerical results}\label{SectionNumerics}

In this section, we illustrate the results obtained above. In the first series of experiments, we work in the geometry 
\[
\Om^{\eps}:=\R\times(0;1)\setminus \left([-0.5;0.5]\times [0.35+\eps;0.65+\eps]\ \cup\  [0;0.5]\times [0.15+\eps;0.85+\eps]\right)
\]
pictured in Figure \ref{GeomTrapped} left. In $\Om:=\Om^0$ the obstacle is symmetric with respect to the line $\R\times\{1/2\}$. According to the results of the literature (see e.g. \cite{EvLV94}), we know that there are trapped modes for certain real frequencies in this geometry. Using Perfectly Matched Layers \cite{Bera94,BeBL04,BoCP18}, we find that they exist for $\sqrt{\lambda^0}\approx1.9939$. Figure \ref{GeomTrapped} right represents such a trapped mode in $\Om$.\\

\begin{figure}[!ht]
\centering
\begin{tikzpicture}[scale=1.4]
\begin{scope}[xshift=-1.4cm].
\draw[fill=gray!30,draw=none](-1.5,0) rectangle (1.5,1);
\draw[fill=white,draw=none](-0.5,0.35) rectangle (0.5,0.7);
\draw[fill=white,draw=none](0,0.15) rectangle (0.5,0.85);
\draw (-1.5,0)--(1.5,0); 
\draw (-1.5,1)--(1.5,1); 
\draw (-0.5,0.35)--(0,0.35)--(0,0.15)--(0.5,0.15)--(0.5,0.85)--(0,0.85)--(0,0.7)--(-0.5,0.7)--(-0.5,0.35); 
\draw[dashed] (1.8,1)--(1.5,1); 
\draw[dashed] (-1.8,1)--(-1.5,1); 
\draw[dashed] (1.8,0)--(1.5,0); 
\draw[dashed] (-1.8,0)--(-1.5,0); 
\draw[dotted,>-<] (0.2,-0.05)--(0.2,0.6);
\node at (0.5,0.3){\footnotesize  $0.5+\eps$};
\end{scope}
\end{tikzpicture}\qquad\quad\includegraphics[width=0.6\textwidth]{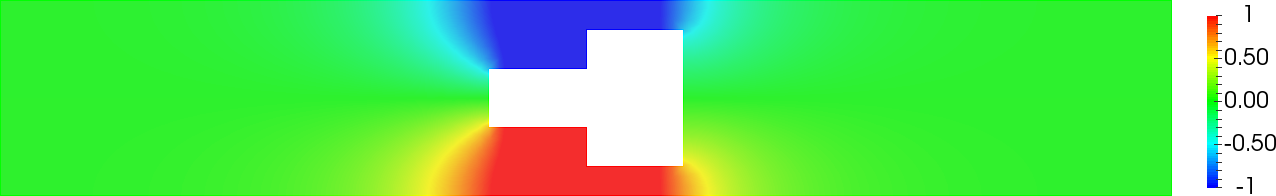}
\caption{Left: geometry of $\Om^{\eps}$. Right: real part of a trapped mode for $\eps=0$ and $\sqrt{\lambda^0}\approx1.9939 $. 
\label{GeomTrapped}}
\end{figure}
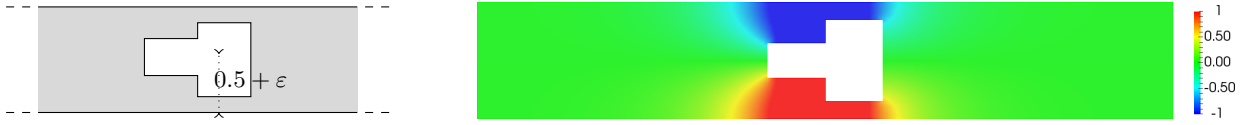

\noindent The domain $\Om^{\eps}$ is obtained from $\Om$ by shifting by $\eps$ the obstacle along the $(Oy)$ axis. Admittedly, this kind of perturbation is not exactly the one considered in (\ref{DefGeom}). However, since there exists an almost identical mapping from $\Om$ to $\Om^{\eps}$, results are similar. We emphasize that for $\eps>0$, $\Om^{\eps}$ has no symmetry property. In Figure \ref{MatriceScattering}, we display the values of the complex scattering coefficients $R_+(\eps,\lambda)$, $T(\eps,\lambda)$ appearing in the decomposition (\ref{defusca}) of $u_{+}$ for $\eps=0.05$ and for $\sqrt{\lambda}\in(1.97;2.03)$ (note that this interval contains the value $\sqrt{\lambda^0}$). To proceed, we use a $\mrm{P2}$ finite element method in a truncated geometry. On the artificial boundary created by the truncation, a Dirichlet-to-Neumann operator with $20$ terms serves as a transparent condition. As expected, we observe that $\lambda\mapsto T(\eps,\lambda)$ passes through zero.\\

\begin{figure}[!ht]
\centering
\includegraphics[scale=0.7]{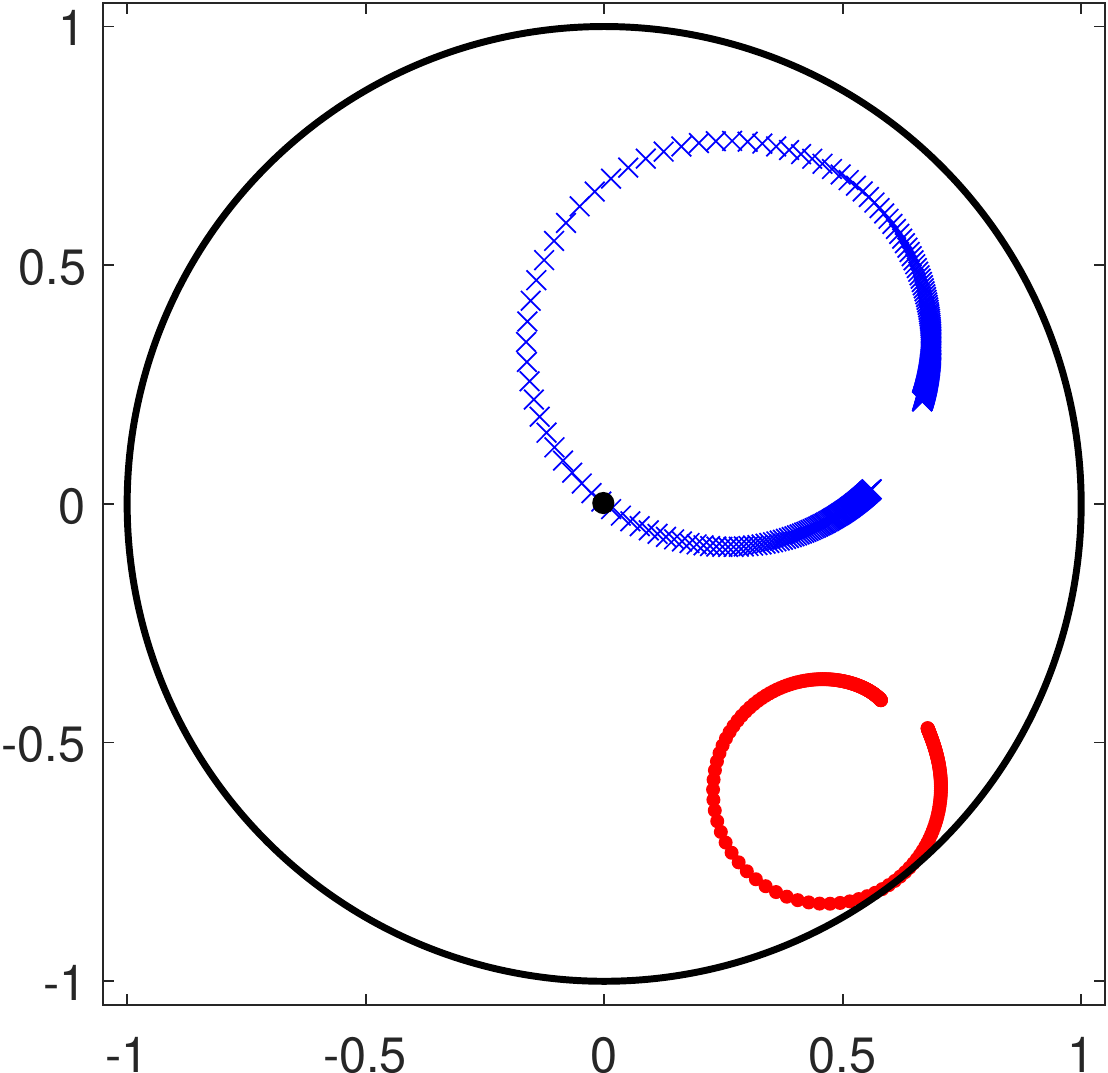}
\caption{Scattering coefficients $T(\eps,\lambda)$ (\textcolor{blue}{$\times$}) and $R_+(\eps,\lambda)$ (\raisebox{0.5mm}{\protect\tikz \fill[red] (0,3) circle (0.6mm);}) for $\eps=0.05$ and $\sqrt{\lambda}\in(1.97;2.03)$. As predicted, $\lambda\mapsto T(\eps,\lambda)$  passes through zero around $\lambda^0$. According to the conservation of energy, we have $|R_+(\cdot,\cdot)|^2+|T(\cdot,\cdot)|^2=1$ and so the scattering coefficients are located inside the unit disk delimited by the black bold line. 
\label{MatriceScattering}}
\end{figure}

\noindent In Figure \ref{FigMultiEps}, we display the curves $\lambda\mapsto |T(\eps,\lambda)|$ for several $\eps$ and a range of values of $\lambda$. The right picture is a zoom of the left picture around $\lambda^0$. As expected we observe that for the different $\eps$, we have $T(\eps,\lambda)=0$ for one $\lambda$ close to $\lambda^0$. We also note that the smaller $\eps>0$ is, the faster the Fano resonance phenomenon occurs. This is also expected. Finally, in Figure \ref{FigTNull}, we display the real part of $u_{+}$ (see (\ref{defusca})) in $\Om^{\eps}$ for $\eps=0.05$ and $\sqrt{\lambda}=2.0072$. In this setting, there holds $T(\eps,\lambda)\approx0$. And indeed we observe that the incident rightgoing wave $w_+$ is completely backscattered, this is the mirror effect.

\begin{figure}[!ht]
\centering
\includegraphics[width=0.45\textwidth]{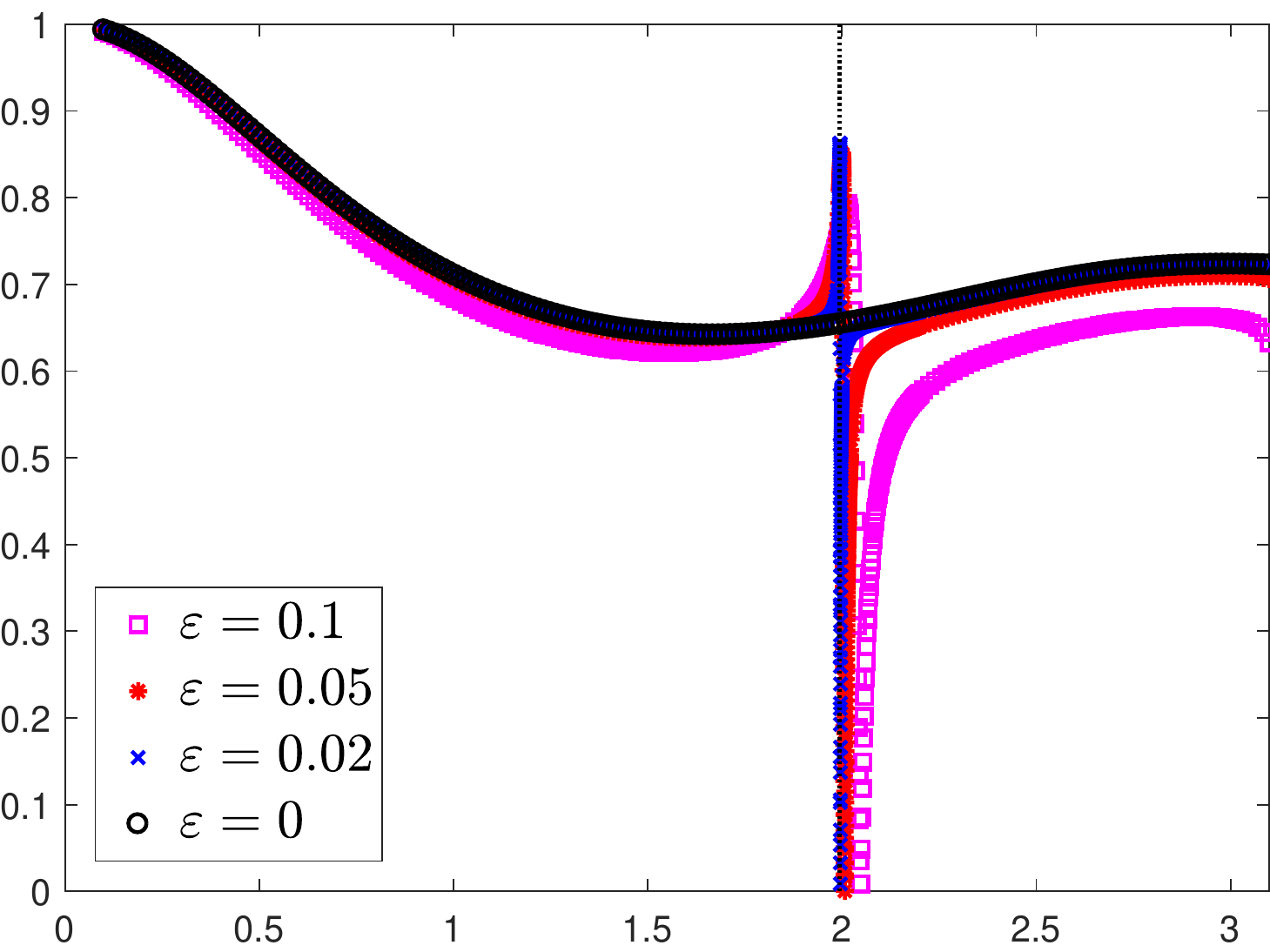}\qquad\includegraphics[width=0.45\textwidth]{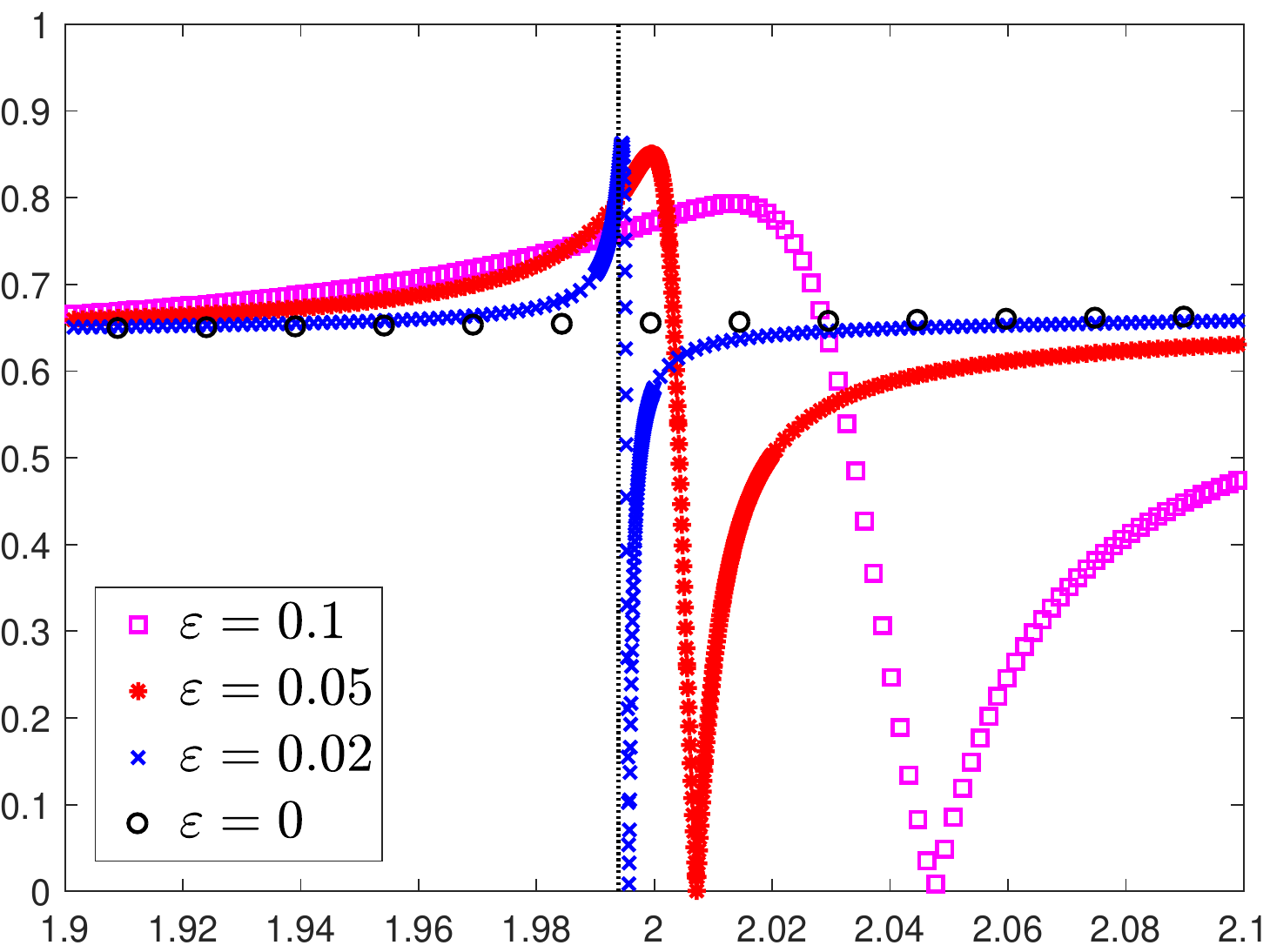}
\caption{Curves $\lambda\mapsto |T(\eps,\lambda)|$ for several $\eps$ and $\sqrt{\lambda}\in(0;\pi)$ (left), $\sqrt{\lambda}\in(1.9;2.1)$ (right). The vertical dotted line represents the value of $\sqrt{\lambda}^0$. \label{FigMultiEps}}
\end{figure}

\begin{figure}[!ht]
\centering
\includegraphics[width=0.9\textwidth]{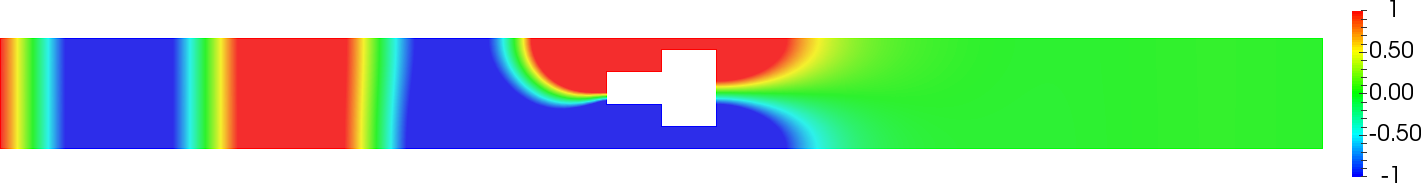}
\caption{Real part of $u_+$ in $\Om^{\eps}$ for $\eps=0.05$ and $\sqrt{\lambda}=2.0072$. In this setting, we have $T(\eps,\lambda)\approx0$. \label{FigTNull}}
\end{figure}

\newpage

\noindent In the second series of experiments, we work in the geometry of Figure \ref{FigTNullBis}. Using Perfectly Matched Layers, we find a complex resonance $\lambda_{c}$ such that $\sqrt{\lambda_{c}}\approx 2.49-0.15i$. In Figure \ref{MatriceScatteringBis}, we display the values of the complex scattering coefficients $R_+(\lambda)$, $T(\lambda)$ appearing in the decomposition (\ref{defusca}) of $u_{+}$ for $\sqrt{\lambda}\in(2.1;2.8)$ (note that this interval contains the value $\Re e\,\sqrt{\lambda_{c}}$). Though this experiment does not strictly enter the framework presented in this note (we do not start from a situation where trapped modes exist), we observe that the curve $\lambda\mapsto T(\lambda)$ passes through zero for $\lambda$ in a neighbourhood of 
$\Re e\,\lambda_{c}$. In Figure \ref{FigTNullBis}, we display the real part of $u_{+}$ for $\sqrt{\lambda}=2.4016$. In this setting, we have $T(\lambda)\approx0$.

\begin{figure}[!ht]
\centering
\includegraphics[scale=0.7]{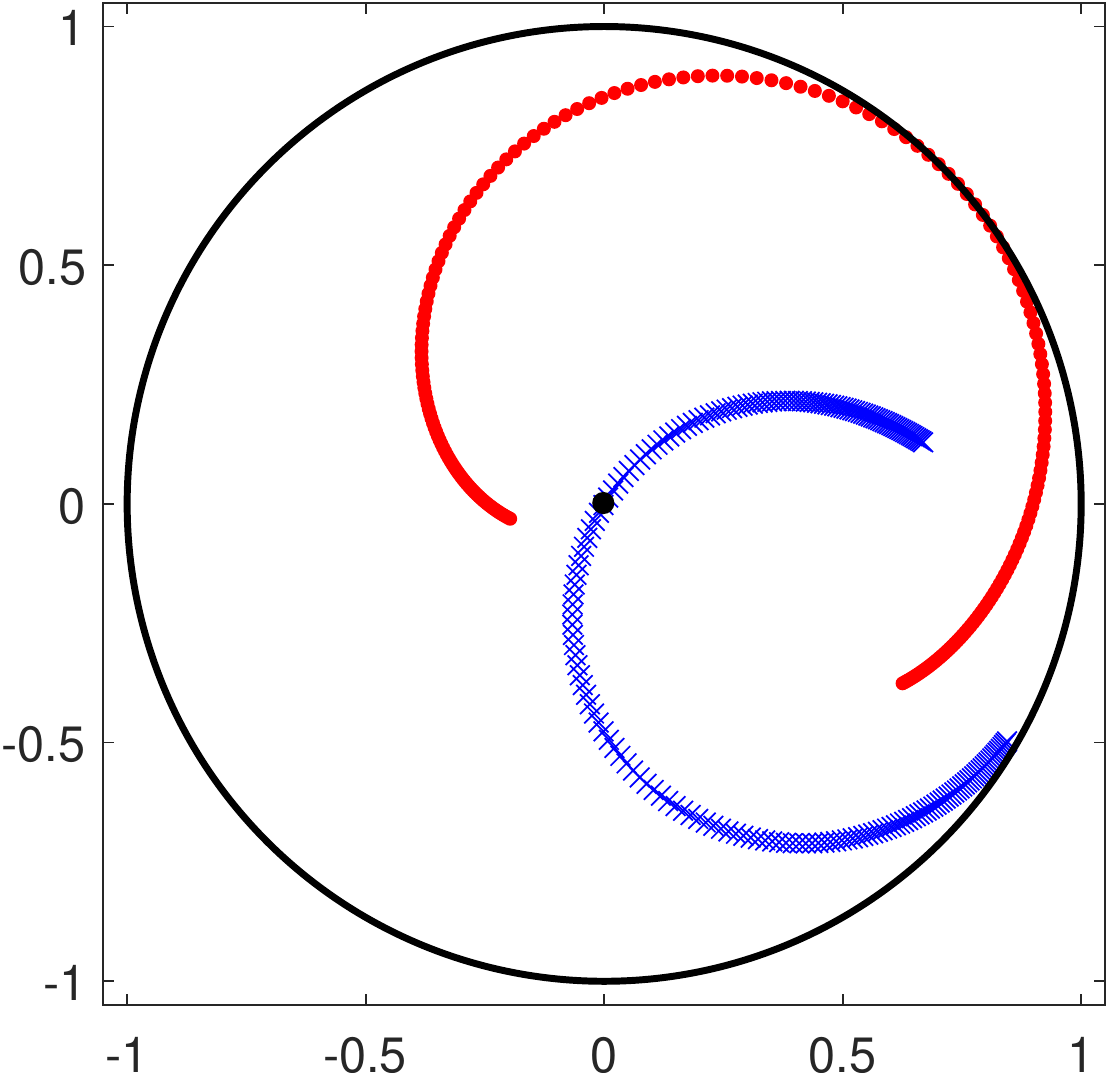}
\caption{Scattering coefficients $T(\lambda)$ (\textcolor{blue}{$\times$}) and $R_+(\lambda)$ (\raisebox{0.5mm}{\protect\tikz \fill[red] (0,3) circle (0.6mm);}) for $\sqrt{\lambda}\in(2.1;2.8)$ in the geometry of Figure \ref{FigTNullBis}. \label{MatriceScatteringBis}}
\end{figure}

\begin{figure}[!ht]
\centering
\includegraphics[width=0.9\textwidth]{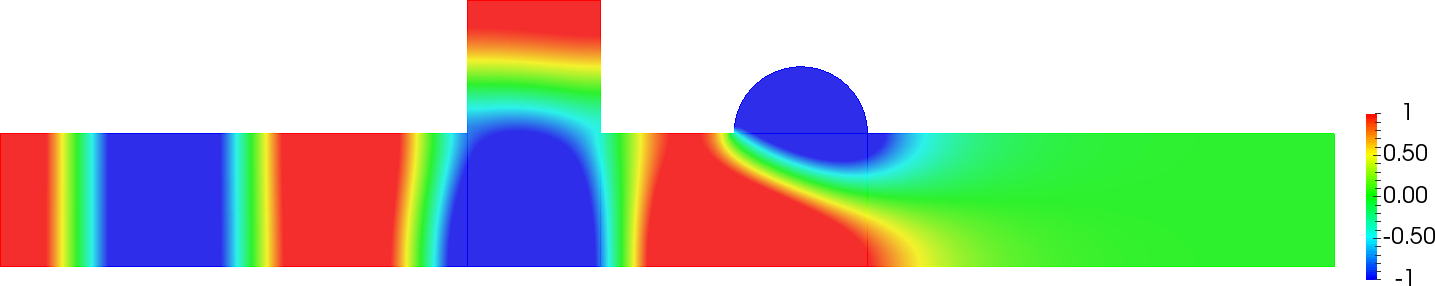}
\caption{Real part of $u_+$ for $\sqrt{\lambda}=2.4016$. In this setting, we have $T(\lambda)\approx0$. \label{FigTNullBis}}
\end{figure}

\newpage

\section{Concluding remarks}
In this note, we proved that during the Fano resonance phenomenon in monomode regime, without assumption of symmetry of the geometry, the transmission coefficient passes through zero. Physically, when the transmission coefficient is null, the energy of an incident wave propagating through the structure is completely backscattered. As already mentioned, everything presented here is also valid in higher dimension and with Dirichlet or periodic boundary conditions instead of Neumann ones. We considered a geometrical perturbation of the walls of the waveguide. We could also have worked with a penetrable inclusion placed in the waveguide. Then perturbing the material parameter, we would have obtained similar results. Importantly, the above analysis applies only in monomode regime, that is for our geometry when $\lambda^0$ belongs to $(0;\pi^2)$. It is not clear what happens in multimodal regime ($\lambda^0>\pi^2)$. Moreover, we assumed that $\lambda^0$ is a simple eigenvalue embedded in the continuous spectrum of the Neumann Laplacian. When $\lambda^0$ is not simple, the analysis has to be done.

\section*{Acknowledgments}
The research of S.A. Nazarov was supported by the grant No. 17-11-01003 of the Russian Science Foundation.

\bibliography{Bibli}
\bibliographystyle{plain}

\end{document}